\documentclass{article}
\pdfoutput=1

\usepackage[title]{appendix}
\usepackage[left=3cm,right=2cm,
    top=2cm,bottom=2cm]{geometry}
\usepackage{amsthm, amssymb,amsmath}
\usepackage{url}\usepackage{hyperref}
\usepackage{authblk}

\newtheorem{theorem}{Theorem}
\newtheorem{corollary}{Corollary}

\newtheorem{remark}{Remark}
\newtheorem{lemma}{Lemma}
\theoremstyle{definition}
\newtheorem{example}{Example}
\theoremstyle{definition}

\numberwithin{theorem}{section}
\numberwithin{equation}{section}
\numberwithin{lemma}{section}
\numberwithin{remark}{section}
\numberwithin{corollary}{section}
\numberwithin{definition}{section}
\numberwithin{example}{section}
\numberwithin{proposition}{section}

\begin{document}

	\title{Implicit Linear Difference Equation over Residue Class Rings}

\author{M.V. Heneralov \thanks{me2001.com@gmail.com} }
\author{A.L. Piven'\thanks{aleksei.piven@karazin.ua} }
\affil{Department of Mathematics \& Computer Sciences\\
	 V. N. Karazin Kharkiv National University}

	\date{}
		\maketitle
		
		\begin{abstract}
			We investigate the first order implicit linear difference equation over
residue class rings modulo $m$. We prove an existence criterion and establish the amount of solutions for this equation.
We obtain analogous results for the initial problem of the considered equation.
The examples which illustrate the developed theory are given.
			
		\end{abstract}

\section{Introduction}

The theory of the linear difference equations is an important branch of mathematics, having a series of different applications (see, for example, \cite{HW}--\cite{Ca}). The theory of implicit linear difference equations in vector spaces was developed in the 80s--90s of the 20 century (see, for example, \cite{Ca}--\cite{BR}).  Unlike the classical theory, the non-invertible operators
have an important role in the new theory. Therefore the interesting problem of the investigation of the implicit linear difference equation with non-invertible coefficients from the any commutative ring appeared. At the moment implicit difference equations over the ring of integers were studied more detailed \cite{GGG}--\cite{GP}. In \cite{mag}  these equations in different classes of topological vector spaces were investigated.

In this paper the first order implicit linear difference equations over residue classes rings is investigated. Let $\mathbb{Z}_m=\mathbb{Z}/m\mathbb{Z}$ be the residue class ring modulo $m$,
where $m\in\mathbb{N}$, $m\ge 2$.
Let $A, B, Y_0\in\mathbb{Z}_m$ and let $\{F_n\}_{n=0}^{\infty}$ be a sequence of $\mathbb{Z}_m$. Consider the initial problem
\begin{equation}\label{1}
	BX_{n+1}=AX_n+F_n,\quad n\in\mathbb{Z}_+,
\end{equation}
\begin{equation}\label{ic}
	X_0=Y_0,
\end{equation}
where $\mathbb{Z}_+$ denotes the set of non-negative integers. The  sequence $\{X_n\}_{n=0}^{\infty}$ of elements of $\mathbb{Z}_m$ is called a solution of the initial problem \eqref{1}, \eqref{ic}, if it satisfies  Equation~\eqref{1} and the initial condition~\eqref{ic}. Equation~\eqref{1} is called implicit, if $B$ is a non-invertible element of the ring $\mathbb{Z}_m$.  If $B$ is an invertible element of $\mathbb{Z}_m$, then this equation  is called explicit. Let $a,b$ are representatives of classes $A,B$ respectively. In the Section 2 we prove that if the greater common divisor of numbers $a,b,m$ is equal to 1, then Equation \eqref{1} is decomposed to the explicit equation \eqref{2} and the implicit equation \eqref{3} which has a unique solution (see lemmas \ref{lemma:isom}, \ref{lemels} and Theorem \ref{2n3:solv}). Theorem \ref{2n3:solv} also gives the general solution for these equations.
The main results of this paper are represented in Section 3 (see theorems \ref{solvability:ip} and \ref{solvability:equation}). Theorem \ref{solvability:ip} describes  necessary and sufficient conditions for the solvability, an amount of solutions and the general solution for the initial problem \eqref{1}, \eqref{ic}. This theorem gives the full description of all possible situations for the initial problem \eqref{1}, \eqref{ic}. The analogous results for Equation \eqref{1} are established in Theorem \ref{solvability:equation}.  This theorem leads to the criteria of the existence and uniqueness of a solution for Equation~\eqref{1} (see Corollaries \ref{solv:cor:2}, \ref{solv:cor:1}). As in the Fredholm theory (see, for example, \cite[Chapter 7]{DS}), Corollary \ref{solv:cor:4} shows that if corresponding to \eqref{1} homogeneous equation has only trivial solution then for any  sequence $\{F_n\}_{n=0}^{\infty}$ of $\mathbb{Z}_m$ Equation \eqref{1} has a unique solution. Section 4 of the present paper contains the examples, which illustrate the constructed theory (see Examples \ref{ex:1}--\ref{ex:4}).

Through this paper  $[t]_s$ denotes the class of the element $t\in\mathbb{Z}$ of the ring $\mathbb{Z}_s$, where $s\in\mathbb{N}$. The ring $\mathbb{Z}_1$ means as the null ring. For the numbers $n_1, n_2, \ldots, n_N\in\mathbb{Z}$ such that $|n_1|+|n_2|+ \ldots +|n_N|\neq 0$ the symbol ${\rm gcd}\left(n_1, \ldots, n_N\right)$ denotes their positive greater common divisor. If $T$ is a nilpotent element of the ring $\mathbb{Z}_s$, then ${\rm ind}(T)$ denotes the nilpotency index of $T$.	

\section{Preliminary}\label{sec:1}

Through this paper $m\ge 2$, $m\in\mathbb{N}$. Let  $A, B$ and $F_n\ (n\in\mathbb{Z}_+)$ be given elements of the ring $\mathbb{Z}_m$.
For each of elements $A, B,  Y_0, F_n, X_n\in \mathbb{Z}_m\ (n\in\mathbb{Z}_+)$ denote, respectively, their representatives
$a, b, f_n, y_0, x_n$.

By Fundamental Theorem of Arithmetic, there exists  pairwise different primes $p_1,\ldots,p_r$
and numbers $k_1,\ldots,k_r\in\mathbb{N}$ such that
$\displaystyle m=\prod_{j=1}^r p_j^{k_j}$.

Denote
$$
	m_1 = \prod_{j\colon p_j\nmid b} p_j^{k_j},\quad 	m_2 =  \prod_{j\colon p_j\mid b} p_j^{k_j},
$$
where $m_1=1$ in the case $p_j\mid b\ (j=1,\ldots,r)$ and $m_2=1$ in the case $p_j\nmid b\ (j=1,\ldots,r)$.
Obviously, $m_1\cdot m_2=m$, and ${\rm gcd}\left(m_1, m_2\right)=1$.

Introduce the natural projections $	\pi_i\colon \mathbb{Z}_m\to\mathbb{Z}_{m_i}$,
defined as follows:
$$
	\pi_i\left(T\right)=[t]_{m_i}, \ \forall T=[t]_{m},\quad
	i=1, 2.
$$
(see \cite[p. 381--382]{D}).

For each $i=1,2$, according to the \cite[p. 381--382]{D} the natural projections $\pi_i\ (i=1,2)$ are homomorphisms.

Denote
\begin{equation*}
	\begin{array}{l}
		A_i = \pi_i\left(A\right),\quad
        B_i = \pi_i\left(B\right),\quad
        Y_{i,0} = \pi_i\left(Y_0\right),\quad
				F_{i, n} = \pi_i\left(F_n\right),\quad 		i=1, 2.
	\end{array}
\end{equation*}

Let $m_1\not=1$, $m_2\not=1$. Introduce the  isomorphism

$$
	\psi\colon \mathbb{Z}_{m_1}\oplus\mathbb{Z}_{m_2}\to\mathbb{Z}_m,
$$
defined as follows (see, for example, \cite[Section 7.6 and Exercise 5 to the Section 7.6]{D}):
\begin{equation}\label{psi:def}
	\psi\left(T_1, T_2\right) = [
		t_1e_1 m_2 +
		t_2e_2 m_1]_m
,\quad
	\forall T_1=[t_1]_{m_1}, \ \forall T_2=[t_2]_{m_2},
\end{equation}
where
\begin{equation}\label{psi:Ei:Def}
	E_1=[e_1]_{m_1}=[m_2]_{m_1}^{-1},\quad E_2=[e_2]_{m_2}=[m_1]_{m_2}^{-1}.
\end{equation}

We regard that since ${\rm gcd}\left(m_1, m_2\right)=1$, the inverse elements $E_1$ and $E_2$ are defined.

If $T_1\in\mathbb{Z}_{m_1}, T_2\in\mathbb{Z}_{m_2}$, then definition of $\psi$ implies
\begin{equation}\label{pi:of:psi}
	\pi_i\left(\psi\left(T_1, T_2\right)\right)=T_i,
	\quad i=1,2.
\end{equation}

 Also, $\pi_1^{-1}\left(T_1\right)\cap\pi_2^{-1}\left(T_2\right)$ is a one-element set, $\psi\left(T_1, T_2\right)$ is an element of this set. This means that
\begin{equation}\label{ch:cor:isom}
	\displaystyle \{\psi\left(T_1, T_2\right)\} = \pi_1^{-1}\left(T_1\right) \cap \pi_2^{-1}\left(T_2\right).
\end{equation}

Consider the following equations over rings $\mathbb{Z}_{m_1}$ and $\mathbb{Z}_{m_2}$ respectively:
\begin{equation}\label{2}
	B_1 X_{1,n+1} = A_1 X_{1,n} + F_{1,n}, \quad n\in\mathbb{Z}_+,
\end{equation}
\begin{equation}\label{3}
	B_2 X_{2,n+1} = A_2 X_{2,n} + F_{2,n}, \quad n\in\mathbb{Z}_+.
\end{equation}

The following lemma describes the connection between solutions of Equation \eqref{1} and  equations~\eqref{2}, \eqref{3}.

\begin{lemma}\label{lemma:isom}
	Let $m_1\neq 1, m_2\neq 1$. The sequence
	\begin{equation}\label{xnpsidef}
	  X_n=\psi\left(X_{1,n}, X_{2,n}\right), \quad n\in\mathbb{Z}_+,
	\end{equation}
	is a solution of Equation~\eqref{1} iff the sequences $\{X_{1,n}\}_{n=0}^{\infty}$ and $\{X_{2,n}\}_{n=0}^{\infty}$ are solutions of equations~\eqref{2}, \eqref{3}, respectively.	
	Moreover, $X_{i,n}=\pi_i\left(X_n\right)$, $i=1,2,\ n\in\mathbb{Z}_+$.
\end{lemma}

\begin{proof}
    The equalities~\eqref{xnpsidef} and~\eqref{pi:of:psi} yield together the equality for $X_{i,n}$: $\pi_i\left(X_n\right)=X_{i,n}$, $i=1,2$.
	
Since $\pi_i\ (i=1,2)$ are homomorphisms, by the equality~\eqref{xnpsidef},
    $$
        \pi_i\left(BX_{n+1} - AX_n - F_n\right)=B_iX_{i,n+1} - A_iX_{i,n} - F_{i,n},\quad i=1,2, \quad n\in\mathbb{Z}_+.
    $$

By the equality~\eqref{ch:cor:isom}, we obtain:

\begin{equation}\label{eququ}
        BX_{n+1} - AX_n - F_n
        =\psi\left(B_1X_{1,n+1} - A_1X_{1,n} - F_{1,n}, B_2X_{2,n+1} - A_2X_{2,n} - F_{2,n}\right),\quad n\in\mathbb{Z}_+.
    \end{equation}
	
We note that
    \begin{equation}\label{pi:i:is:0}
        \pi_1\left(0\right)=0\mbox{ and }\pi_2\left(0\right)=0.
    \end{equation}
    Since~\eqref{pi:i:is:0}, \eqref{eququ} hold, we obtain that the equality \eqref{1} is fulfilled if and only if  equalities \eqref{2}, \eqref{3} are fulfilled. This ends the proof of the lemma.
\end{proof}

Introduce the notation:
$$
	d={\rm gcd}\left(a, b, m\right).
$$

Consider the equations~\eqref{2} and~\eqref{3}. The following lemma establishes important  properties for coefficients of these equations.

\begin{lemma}\label{lemels}
	The following statements hold.
	\begin{enumerate}
		\item Let $m_1\neq 1$, then $B_1$ is invertible.
		\item Let $m_2\neq 1$, then $B_2$ is nilpotent. If additionally $d=1$, then $A_2$ is invertible.
        \item $B$ is nilpotent if and only if $m_1=1$.
	\end{enumerate}
\end{lemma}

\begin{proof}
	Proof the statement~1. The definition of $m_1$ and $B$ implies the equality ${\rm gcd}\left(b, m_1\right)=1$. Hence, the element $B_1$ is an invertible element of $\mathbb{Z}_{m_1}$.
	
	Proof the statement~2.    Firstly prove that $B_2$ is nilpotent. Is it evident from the definition of $m_2$: if $k=\displaystyle\max_{j=1,\ldots,r}\{k_j\}$, then $B_2^k=[b^k]_{m_2}=0$.

    Let $d=1$. We will prove that ${\rm gcd}\left(a, m_2\right)=1$. Assuming the contrary, we obtain that there exists $j\in \{1,\ldots,r\}$ such that $p_j\mid a$. This condition yields $p_j\mid a, p_j\mid b, p_j\mid m$. Hence $p_j\mid d$. But it contradicts $d=1$. Therefore ${\rm gcd}\left(a, m_2\right)=1$. This means that $A_2$ is an invertible element of $\mathbb{Z}_{m_2}$.

    Proof the statement~3. The condition $m_1=1$ is equivalent to the assertion
    $$
        \forall j=1,\ldots,r\colon p_j\mid b.
    $$

    The last condition is equivalent to the nilpotency of the element $B$ in $\mathbb{Z}_m$.
\end{proof}

\begin{remark}\rm
Lemma \ref{lemels} is an analogue of the spectral decomposition of a regular  operator pencil in Banach spaces (see \cite[Lemma 2.1]{R}). The analogous to \eqref{2}, \eqref{3} decomposition of an implicit difference equation in Banach spaces into two equations  with regarded properties was obtained in \cite{BR,BV}.
\end{remark}

The following theorem is a solvability theorem for Equation~\eqref{2} and, in the case $d=1$, for Equation~\eqref{3}.

\begin{theorem}\label{2n3:solv}
	The following statements hold.
	
	\begin{enumerate}
		\item Let $m_1\neq 1$. The general solution of Equation~\eqref{2} is defined by the following formula:
		\begin{equation}\label{2:sol}
			X_{1,n} = B_1^{-n}A_1^nX_{1,0} + \sum_{s=0}^{n-1}A_1^sB_1^{-s-1}F_{1,n-s-1}, \quad n\in\mathbb{N},
		\end{equation}
		where $X_{1,0}$ is an arbitrary element of $\mathbb{Z}_{m_1}$.
		\item Let $d=1$ and $m_2\neq 1$. Then Equation~\eqref{3} has a unique solution, defined by the following formula:
		\begin{equation}\label{3:sol}
			X_{2,n} = - \sum_{s=0}^{{\rm ind}\left(B_2\right)-1}A_2^{-s-1}B_2^sF_{2,n+s}, \quad n\in\mathbb{Z}_+.
		\end{equation}
	\end{enumerate}
\end{theorem}

\begin{remark}\rm	The corresponding inverse  elements  exist according to Lemma~\ref{lemels}.
\end{remark}

\begin{proof}
	Prove firstly the statement~1.	
	By the statement~1 of Lemma~\ref{lemels}, $B_1$ is invertible. The equality~\eqref{2} is equivalent to the equality
	\begin{equation}\label{rec:labelproof23:1}
		X_{1,n+1} = B_1^{-1}A_1X_{1,n} + B_1^{-1}F_{1,n}, \quad n\in\mathbb{Z}_+.
	\end{equation}
	
	According to~\cite[p.~4]{Ela}, the general solution of Equation~\eqref{rec:labelproof23:1} has the form~\eqref{2:sol}.
	
	Prove now the statement~2.
	By the statement~2 of Lemma~\ref{lemels}, $A$ is an invertible element, and $B$ is nilpotent (these are since $d=1$).
	
	The equality~\eqref{3} is equivalent to the following:
	
	\begin{equation}\label{rec:labelproof23:2}
		X_{2,n} = - A_2^{-1}F_{2,n} + A_2^{-1}B_2X_{2,n+1}, \quad n\in\mathbb{Z}_+.
	\end{equation}
	
	Applying~\eqref{rec:labelproof23:2} recurrently few times, obtain the equality~\eqref{3:sol}.
	
    Now let $\{X_n\}_{n=0}^{\infty}$ be defined by the formula~\eqref{3:sol}. Denote $k={\rm ind}\left(B_2\right)$. Substituting~\eqref{3:sol} to the left part of Equation~\eqref{3}, we obtain:
	\begin{multline*}
		B_2X_{n+1}=
		-B_2A_2^{-1}\sum_{s=0}^{k-1}A_2^{-s}B_2^sF_{2,n+1+s} =
		- \sum_{s=0}^{k-1}A_2^{-s-1}B_2^{s+1}F_{2,n+s+1} = - \sum_{t=1}^kA_2^{-t}B_2^tF_{2,n+t}=\\
		= - \sum_{t=0}^kA_2^{-t}B_2^tF_{2,n+t} + F_{2,n} =- A_2\cdot A_2^{-1}\sum_{t=0}^{k-1}A_2^{-t}B_2^tF_{2,n+t} + F_{2,n} =
		A_2X_{2,n}+F_{2,n}.
\end{multline*}

	Therefore $\{X_{2,n}\}_{n=0}^{\infty}$, defined by the formula~\eqref{3:sol}, is the unique solution of Equation~\eqref{3}.
\end{proof}

\begin{corollary}\label{2n3:solv:cor}
	The following statements hold.
	
	\begin{enumerate}
		\item Let $d=1$ and $m_1=1$. Equation~\eqref{1} has a unique solution $\{X_n\}_{n=0}^{\infty}$, defined by the following formula:
		\begin{equation}\label{1:2n3:solv:cor:1}
			X_n = - \sum_{s=0}^{{\rm ind}\left(B\right)-1}A^{-s-1}B^sF_{n+s}, \quad n\in\mathbb{Z}_+.
		\end{equation}
		\item Let $m_2=1$. The general solution of Equation~\eqref{1} is defined by the following formula:
		\begin{equation}\label{1:2n3:solv:cor:2}
			X_n = B^{-n}A^nX_0 + \sum_{s=0}^{n-1}A^sB^{-s-1}F_{n-s-1}, \quad n\in\mathbb{N},
		\end{equation}
		where $X_0$ is an arbitrary element of $\mathbb{Z}_m$.
	\end{enumerate}
\end{corollary}

\section{Main results}\label{sec:2}

Here we obtain the solvability theorems over $\mathbb{Z}_m$ for Equation~\eqref{1} and for the initial problem~\eqref{1}, \eqref{ic}.

Introduce the following notations:
$$
    m'=\frac{m}{d},\
	Y'_0=[y_0]_{m'},\
	A'=[a/d]_{m'},\
	B'=[b/d]_{m'}.
$$
Also, when $d\mid f_n$ for all $n\in\mathbb{Z}_+$, denote
$$
    F'_n=[f_n/d]_{m'}, \quad n\in\mathbb{Z}_+.
$$

Each a prime divisor of the number $m'$ is also a divisor of $m$. Then by Fundamental Theorem of Arithmetic, there exist non-negative integers $l_j\le k_j$ ($j=1,\ldots,r$) such that $\displaystyle m'=\prod_{j=1}^r p_j^{l_j}$.

Denote also
$$
	m'_1 = \prod_{j\colon dp_j\nmid b}p_j^{l_j},\quad
    m'_2=\prod_{j\colon dp_j\mid b}p_j^{l_j},
$$
$$
	A'_i = [a/d]_{m'_i},\quad B_i' = [b/d]_{m'_i},\quad
    Y'_{i,0} = [y_0]_{m'_i},\quad i=1,2.
$$
As in the definition $m_1,m_2$, we assume $m'_1=1$ in the case $dp_j\mid b\ (j=1,\ldots,r)$ and $m'_2=1$ in the case $dp_j\nmid b\ (j=1,\ldots,r)$.
Note that if $d=1$, then $m'_i=m_i$, $i=1,2$.

Let $d\mid f_n$ for all $n\in\mathbb{Z}_+$. Denote
$$
    F'_{i,n} = [f_n/d]_{m_i'}.
$$
and consider the initial problem

\begin{equation}\label{1:dash}
    B'X'_{n+1}=A'X'_n+F'_n, \quad n\in\mathbb{Z}_+,
\end{equation}
\begin{equation}\label{ic:dash}
	X'_0=Y'_0
\end{equation}
over $\mathbb{Z}_{m'}$.

The following statement is a helpful lemma, which shows the connection between the equations \eqref{1} and ~\eqref{1:dash}.

\begin{lemma}\label{hst1}
    Let $d\neq 1$, $d\mid f_n$ ($n\in\mathbb{Z}_+$).  The sequence $\{X_n\}_{n=0}^{\infty}$ is a solution of Equation~\eqref{1} iff it admits the following representation
    \begin{equation}\label{lem:xnlocaldeflem}
        X_n=[x'_n+\alpha_n m']_m,\quad n\in\mathbb{Z}_+,
    \end{equation}
	where $X'_n=[x'_n]_{m'}\ (n\in\mathbb{Z}_+)$ is a solution of Equation~\eqref{1:dash}, and $\{\alpha_n\}_{n=0}^{\infty}$ is a sequence of $\{0,1,\ldots,d-1\}$. Moreover, the sequence $\{\alpha_n\}_{n=0}^{\infty}$ and the solution $\{X'_n\}_{n=0}^{\infty}$ of Equation ~\eqref{1:dash} with $x'_n\in\{0,\ldots,m'-1\}$ are uniquely determined by the solution~$\{X_n\}_{n=0}^{\infty}$ of Equation~\eqref{1}.
\end{lemma}

\begin{proof}

   Obviously, Equation~\eqref{1} is equivalent to the congruence
	
	\begin{equation}\label{trata2}
		bx_{n+1}\equiv ax_n+f_n\pmod m, \quad n\in\mathbb{Z}_+.
	\end{equation}
	The congruence~\eqref{trata2} is equivalent to the following condition.
	
	\begin{equation}\label{trata2balahta}
		\frac{b}{d}x_{n+1}\equiv \frac{a}dx_n+\frac{f_n}d\pmod{m'}, \quad n\in\mathbb{Z}_+.
	\end{equation}

	The congruence~\eqref{trata2balahta} means that there exists
a solution $X'_n=[x'_n]_{m'}\ (n\in\mathbb{Z}_+)$ of Equation ~\eqref{1:dash} such that $x_n\equiv x'_n\pmod {m'}$.
Therefore $\{X_n\}_{n=0}^{\infty}$ is a solution of~\eqref{1} if and only if  $X_n=[x'_n + \alpha_n\cdot m']_m\ (n\in\mathbb{Z}_+)$, where $\{\alpha_n\}_{n=0}^{\infty}$ is an arbitrary sequence of $\{0,\ldots,d-1\}$.

       Suppose that the two following representatives for the solution of Equation \eqref{1} hold:
       $$X_n=[x'_n+\alpha_n m']_m=\left[\widehat{x'_n}+\widehat{\alpha_n}m'\right]_m,\quad n\in\mathbb{Z}_+,$$
       where
       $X'_n=[x'_n]_{m'}$, $\widehat{X'_n}=\left[\widehat{x'_n}\right]_{m'} (n\in\mathbb{Z}_+)$ are solutions of Equation ~\eqref{1:dash}, $\alpha_n,\ \widehat{\alpha_n}\ (n\in\mathbb{Z}_+)$ are numbers from $\{0,\ldots,d-1\}$ and additionally $x_n, \widehat{x_n}\in\{0,\ldots,m'-1\}$. It implies the following congruence
                 \begin{equation}\label{xnxnxn}
    	x'_n+\alpha_n m'
    	\equiv
    	\widehat{x'_n}+\widehat{\alpha_n}m'\pmod{m},\quad   	 n\in\mathbb{Z}_+.
    \end{equation}
	Then $x'_n\equiv \widehat{x'_n}\pmod{m'}$. According to the assumption $\widehat{x'_n}, x'_n\in\{0,\ldots,m'-1\}$, we have
 $x'_n=\widehat{x'_n},\ n\in\mathbb{Z}_+$. Now the congruence~\eqref{xnxnxn} means $\alpha_n\equiv \widehat{\alpha_n}\pmod{d}$. Since $\alpha_n,\widehat{\alpha_n}\in\{0,\ldots,d-1\}$, we have  $\alpha_n = \widehat{\alpha_n},\ n\in\mathbb{Z}_+$.
\end{proof}

The following theorem is a solvability theorem for the initial problem~\eqref{1}, \eqref{ic}. This theorem also establishes the explicit form for the general solution of the considered initial problem, when a solution exists.

\begin{theorem}\label{solvability:ip}
		The following statements hold.
	
	\begin{enumerate}
        \item The initial problem ~\eqref{1}, \eqref{ic} has a unique solution iff $d=1$ and one of the following conditions holds:
            \begin{enumerate}
            \item $m_2=1$;
            \item $m_2\neq 1$ and the equality
		\begin{equation}\label{3:y0}
			Y_{2,0} = - \sum_{s=0}^{{\rm ind}(B_2)-1}A_2^{-s-1}B_2^sF_{2, s}
		\end{equation}
		is fulfilled.
\end{enumerate}
Moreover, the unique solution of the initial problem~\eqref{1}, \eqref{ic} is defined by the formula
		\begin{equation}\label{solvability:ip:d:eq1:solexpl}
			X_n=\left\{
			\begin{array}{cc}
				\displaystyle
				B^{-n}A^nY_0 + \sum_{s=0}^{n-1}A^sB^{-s-1}F_{n-s-1},\quad  m_2=1,\\
				\displaystyle - \sum_{s=0}^{{\rm ind}(B) -1}A^{-s-1}B^sF_{n+s}, \quad m_1=1,\\
				\displaystyle\psi\left(
					B_1^{-n}A_1^nY_{1,0} + \sum_{s=0}^{n-1}A_1^sB_1^{-s-1}F_{1,n-s-1},
					- \sum_{s=0}^{{\rm ind}(B_2)-1}A_2^{-s-1}B_2^sF_{2,n+s}
				\right), \ m_1\neq 1, m_2\neq 1.
			\end{array}
			\right.
		\end{equation}
		\item The initial problem~\eqref{1}, \eqref{ic} has infinitely many solutions iff $d\neq 1$, $d\mid f_n$ for all $n\in\mathbb{Z}_+$ and one of the following conditions holds:
 \begin{enumerate}
 \item $m'_2=1$;
 \item $m'_2\neq 1$ and the equality
		\begin{equation}\label{3:sol:y0}
			Y_{2,0}' =
            -\sum_{s=0}^{{\rm ind}(B'_2)-1}\left(A'_2\right)^{-s-1}\left(B'_2\right)^sF'_{2,s}
		\end{equation}
		is fulfilled.
\end{enumerate}
		The general solution of the initial problem~\eqref{1}, \eqref{ic} is defined by
		\begin{equation}\label{th1:x'n:to:xn}
		  X_n=[x'_n + \alpha_n\cdot m']_m, \quad n\in\mathbb{N},
		\end{equation}
        where $X'_n=[x'_n]_{m'}\ (n\in\mathbb{Z}_+)$ is a solution of the initial problem~\eqref{1:dash}, \eqref{ic:dash} (this solution exists and is unique), and $\{\alpha_n\}_{n=1}^{\infty}$ is an arbitrary sequence of $\{0,1,\ldots,d-1\}$.
        Moreover, the sequence $\{\alpha_n\}_{n=1}^{\infty}$  is uniquely determined by the solution~$\{X_n\}_{n=0}^{\infty}$ of the initial problem \eqref{1}, \eqref{ic}.
        		\item The initial problem~\eqref{1}, \eqref{ic} has no solutions iff one of the following conditions holds:
        \begin{enumerate}
         \item $d\nmid f_n$ for some $n\in\mathbb{Z}_+$;
         \item $d\mid f_n\ (n\in\mathbb{Z}_+)$, $m'_2\neq 1$ and the equality~\eqref{3:sol:y0} is not fulfilled.
         \end{enumerate}
	\end{enumerate}
\end{theorem}

\begin{remark}\rm
In the statement 2 of Theorem~\ref{solvability:ip} the sequence $\{X'_n\}_{n=0}^{\infty}$, when $m\neq d$, may be defined by the formula, analogous to the formula~\eqref{solvability:ip:d:eq1:solexpl}, applied to the initial problem~\eqref{1:dash}, \eqref{ic:dash}. When $m=d$, then  evidently $X'_n=0$ for all $n\in\mathbb{Z}_+$.
\end{remark}

\begin{proof}
The sufficient conditions of all three statements of Theorem~\ref{solvability:ip} are mutually exclusive and they exhaust all possibilities. Therefore  it is enough to prove the
sufficiency for all of three statements of this theorem.

	Prove the sufficiency of the statement~1. Let $d=1$.
   If either $m_1=1$, or $m_2=1$, then the claimed statement follows from Corollary~\ref{2n3:solv:cor}.
    Further let $m_1\neq 1$ and $m_2\neq 1$.

    Set the initial conditions:
    \begin{equation}\label{2:ic}
        X_{1,0}=Y_{1,0}\in\mathbb{Z}_{m_1},
    \end{equation}
    \begin{equation}\label{3:ic}
        X_{2,0}=Y_{2,0}\in\mathbb{Z}_{m_2}
    \end{equation}
    for equations~\eqref{2} and~\eqref{3} respectively.
	
    According to Lemma~\ref{lemma:isom}, the sequence $\{X_n\}_{n=0}^{\infty}$ is a solution of the initial problem~\eqref{1}, \eqref{ic} if and only if the sequence $\{X_{1,n}\}_{n=0}^{\infty}$ is a solution of the initial problem~\eqref{2}, \eqref{2:ic} and the sequence $\{X_{2,n}\}_{n=0}^{\infty}$ is a solution of the initial problem~\eqref{3}, \eqref{3:ic}.   By the statement~1 of Theorem~\ref{2n3:solv}, the initial problem~\eqref{2}, \eqref{2:ic} has a solution for any $Y_{1,0}\in\mathbb{Z}_{m_1}$. According to the statement~2 of Theorem~\ref{2n3:solv}, the initial problem~\eqref{3}, \eqref{3:ic} has a solution if and only if $Y_{2,0}$
    satisfies~\eqref{3:y0}. Hence, the initial problem~\eqref{1}, \eqref{ic} has a solution if and only if the condition~\eqref{3:y0} is fulfilled, moreover this solution is unique and has the form~\eqref{xnpsidef}, where $X_{1,n}$ and $X_{2,n}$ are defined by the formulas~\eqref{2:sol} and~\eqref{3:sol} respectively.

Prove the sufficiency of the statement~2. Let $d\neq 1$, $d\mid f_n$ for all $n\in\mathbb{Z}_+$. Additionally, let either $m_2'=1$, or $m_2'\neq 1$ and~\eqref{3:sol:y0} be fulfilled.
    Since ${\rm gcd}\left(a/d, b/d, m'\right)=1$, we can apply the sufficiency of the statement~1 (which is already proved) to the initial problem~\eqref{1:dash}, \eqref{ic:dash}.   Due to that statement, the initial problem~\eqref{1:dash}, \eqref{ic:dash} has a unique solution $X'_n=[x'_n]_{m'}\ (n\in\mathbb{Z}_+)$.    By Lemma~\ref{hst1}, for any sequence $\{\alpha_n\}_{n=0}^{\infty}$ of $\{0,\ldots,d-1\}$ the formula~\eqref{lem:xnlocaldeflem} defines the solution of Equation~\eqref{1}.

    We choose $\alpha_0$ such that \eqref{ic} is fulfilled, i.~e.
    $[x'_0+\alpha_0m']_m=[y_0]_m$.    The initial condition \eqref{ic:dash} implies $[x'_0]_{m'}=[y_0]_{m'}$, and the following congruence holds $x'_0\equiv y_0\pmod{m'}$. Then $\beta=\frac{y_0-x'_0}{m'}\in\mathbb{Z}$. Divide $\beta$ on $d$ with remainder. Then there exist $q\in\mathbb{Z}$ and $\alpha_0\in\{0,\ldots,d-1\}$ such that
   $\beta=qd+\alpha_0$. Therefore,
   $$[x'_0+\alpha_0m']_m=[x'_0+(\beta-qd)m']_m=[x'_0+y_0-x'_0-qm]_m=[y_0]_m.$$
    Therefore for the chosen $\alpha_0$ and any sequence $\{\alpha_n\}_{n=1}^{\infty}$ of $\{0,\ldots,d-1\}$ the formula~\eqref{th1:x'n:to:xn} defines a solution of the initial problem~\eqref{1}, \eqref{ic}. By Lemma~\ref{hst1}, the
    expression ~\eqref{th1:x'n:to:xn}  gives infinitely many solutions of this initial problem (see also \eqref{lem:xnlocaldeflem}).

    We prove that the general solution of the initial problem~\eqref{1}, \eqref{ic} is defined by the formula~\eqref{th1:x'n:to:xn}. Let $\{X_n\}_{n=0}^{\infty}$ be an arbitrary solution of this initial problem. Then by Lemma~\ref{hst1}, this solution has the form~\eqref{lem:xnlocaldeflem}, where $\{X'_n\}_{n=0}^{\infty}$ is a solution of Equation~\eqref{1:dash}. Moreover, $\{X'_n\}_{n=0}^{\infty}$ must satisfy the initial condition~\eqref{ic:dash}. We have proved that the initial problem~\eqref{1:dash}, \eqref{ic:dash} has a unique solution. Hence the general solution of the initial problem~\eqref{1}, \eqref{ic} has the form~\eqref{th1:x'n:to:xn}.

    Now prove the sufficiency of the statement~3. Assume $d\nmid f_n$ for some $n\in\mathbb{Z}_+$.
		The equality \eqref{1} for this $n$  is equivalent to the congruence $bx_{n+1}-ax_n\equiv f_n\pmod m$. Hence,
	\begin{equation}\label{fnnotdl}
		f_n\equiv d\cdot\left(\frac {b}{d} x_{n+1}-\frac {a}{d} x_n\right)\pmod m.
	\end{equation}
     Since $d\mid m$, the condition~\eqref{fnnotdl} means $d\mid f_n$, which is a contradiction the assumption.
    Therefore, if $d\nmid f_n$ for some $n\in\mathbb{Z}_+$, then Equation~\eqref{1} has no solutions.
	Now  suppose that $d\neq 1$, $d\mid f_n,\ n\in\mathbb{Z}_+$, $m'_2\neq 1$ and the equality~\eqref{3:sol:y0} is not fulfilled. Assume the contrary, that the initial problem \eqref{1}, \eqref{ic} has a solution $X_n=[x_n]_m\ (n\in\mathbb{Z}_+)$. Then the congruence \eqref{fnnotdl} is fulfilled for all $n\in
\mathbb{Z}_+$ and the sequence $X'_n=[x_n]_{m'}\ (n\in\mathbb{Z}_+)$ is a solution of the initial problem~\eqref{1:dash}, \eqref{ic:dash}. Since ${\rm gcd}\left(a/d, b/d, m'\right)=1$, we can apply the sufficiency of the statement~1 (which is already proved) to this initial problem. Therefore, if $m'_2\neq 1$ and $\{X'_n\}_{n=0}^{\infty}$ is a solution of the initial problem~\eqref{1:dash}, \eqref{ic:dash}, then $Y_{2,0}'=[y_0]_{m'_2}$  must satisfy~\eqref{3:sol:y0}. This contradicts the assumption.
\end{proof}

The following theorem is a solvability theorem for Equation~\eqref{1}. This theorem also yields the explicit form for the general solution of Equation~\eqref{1}.

\begin{theorem}\label{solvability:equation}
	 The following statements hold.
	\begin{enumerate}
		\item Equation~\eqref{1} has a finite amount of solutions iff $d=1$. Moreover, the amount of these solutions is equal to $m_1$ and in this case
		\begin{enumerate}
			\item If $m_2=1$, then the general solution of Equation~\eqref{1} has the form
			\begin{equation}\label{gen1}
				X_n = B^{-n}A^nX_0 + \sum_{s=0}^{n-1}A^sB^{-s-1}F_{n-s-1}, \quad n\in\mathbb{N},
			\end{equation}
			where $X_0$ is an arbitrary element of $\mathbb{Z}_m$.
			\item If $m_1=1$, then the unique solution of Equation~\eqref{1} has the form
			\begin{equation}\label{gen2}
				X_n = - \sum_{s=0}^{{\rm ind}(B) -1}A^{-s-1}B^sF_{n+s},\quad
				n\in\mathbb{Z}_+.
			\end{equation}
			\item If $m_1\neq 1$ and $m_2\neq 1$, then the general solution of Equation~\eqref{1} has the form
			\begin{multline}\label{gen_sol_1}
				X_0 = \psi\left(
					X_{1,0},
					- \sum_{s=0}^{{\rm ind}(B_2)-1}A_2^{-s-1}B_2^sF_{2,s}
				\right),\\
				X_n = \psi\left(
					B_1^{-n}A_1^nX_{1,0} + \sum_{s=0}^{n-1}A_1^sB_1^{-s-1}F_{1,n-s-1},
					- \sum_{s=0}^{{\rm ind}(B_2)-1}A_2^{-s-1}B_2^sF_{2,n+s}
				\right),\quad n\in\mathbb{N},
			\end{multline}
			where $X_{1,0}$ is an arbitrary element of $\mathbb{Z}_{m_1}$.
		\end{enumerate}
		
		\item Equation~\eqref{1} has infinitely many solutions iff $d\neq 1$ and $d\mid f_n$ for all $n\in\mathbb{Z}_+$. The general solution in this case has the form~\eqref{lem:xnlocaldeflem},
		where $X'_n=[x'_n]_{m'}\ (n\in\mathbb{Z}_+)$ is the general solution of Equation~\eqref{1:dash}, and  $\{\alpha_n\}_{n=0}^{\infty}$ is an arbitrary sequence of $\{0,\ldots,d-1\}$.
Moreover, the sequence $\{\alpha_n\}_{n=0}^{\infty}$ and the solution $\{X'_n\}_{n=0}^{\infty}$ of Equation ~\eqref{1:dash} with $x'_n\in\{0,\ldots,m'-1\}$ are uniquely determined by the solution~$\{X_n\}_{n=0}^{\infty}$ of Equation~\eqref{1}.
\item Equation~\eqref{1} has no solutions iff $d\nmid f_n$ for some $n\in\mathbb{Z}_+$.
	\end{enumerate}
\end{theorem}
\begin{remark}\rm
 Since ${\rm gcd}\left(a/d, b/d, m'\right)=1$, in the statement 2 of Theorem~\ref{solvability:equation} the general solution of Equation $\{X'_n\}_{n=0}^{\infty}$ may be defined by the formula, analogous to  formulas~\eqref{gen1}--\eqref{gen_sol_1}, applied to Equation~\eqref{1:dash}.
\end{remark}
\begin{proof}
The sufficient conditions of all three statements of Theorem~\ref{solvability:equation} are mutually exclusive and they exhaust all possibilities. Therefore  it is enough to prove the
sufficiency for all of three statements of this theorem.

We prove the sufficiency of the statement 1 of Theorem \ref{solvability:equation}. Let $d=1$.	If either $m_1=1$ or $m_2=1$, then the claimed statement implies from the corollary~\ref{2n3:solv:cor}.
	Let $m_1\neq 1$ and $m_2\neq 1$. The statement~1 of Theorem~\ref{solvability:ip} implies that if there exists a solution of the initial problem~\eqref{1}, \eqref{ic}, then it is defined uniquely by the given $Y_{1,0}$, where $Y_{1,0}$ is an arbitrary element of the ring $\mathbb{Z}_{m_1}$. Therefore, the amount of solutions of  Equation ~\eqref{1}  is equal to $m_1$. The form \eqref{gen_sol_1} of the general solution of Equation \eqref{1} is obtained with the help of the general solution  \eqref{solvability:ip:d:eq1:solexpl} of the initial problem \eqref{1}, \eqref{ic}.

	Prove the sufficiency of the statement~2 of Theorem \ref{solvability:equation}. Let $d\neq 1$ and $d\mid f_n$ for all $n\in\mathbb{Z}_+$. Since ${\rm gcd}\left(\frac{a}d, \frac{b}d, m'\right)=1$,
	we can apply the sufficiency of the statement~1 (which is already proved) to Equation~\eqref{1:dash}. Due to that statement, Equation~\eqref{1:dash} has $m'_1$ solutions. Let $X'_n=[x'_n]_{m'}$ ($n\in\mathbb{Z}_+$) be the general solution of this equation. By Lemma \ref{hst1}, the general solution of Equation~ \eqref{1} has the form~\eqref{lem:xnlocaldeflem}, where $\{\alpha_n\}_{n=0}^{\infty}$ is an arbitrary sequence of $\{0,\ldots,d-1\}$. Moreover, by Lemma \ref{hst1}, Equation~\eqref{1} has infinitely many solutions.

	Prove the sufficiency of the statement~3 of Theorem \ref{solvability:equation}. Let $d\neq 1$ and $d\nmid f_n$ for some $n\in\mathbb{Z}_+$. By the statement~3 of Theorem~\ref{solvability:ip}, for any $Y_0\in\mathbb{Z}_m$ the initial problem~\eqref{1}, \eqref{ic} has no solutions. Hence, Equation~\eqref{1} has no solutions.
\end{proof}

The following corollary of Theorem~\ref{solvability:equation} yields the solvability of Equation \eqref{1} in the case of an invertible element $A$.

\begin{corollary}
If $A$ is an invertible element of $\mathbb{Z}_m$, then Equation \eqref{1} always has a solution. Moreover, the amount of solutions for Equation \eqref{1} is equal to $m_1$.
\end{corollary}

Theorem~\ref{solvability:equation} also implies the following criteria of the existence and uniqueness of a solution for Equation~\eqref{1}.

\begin{corollary}\label{solv:cor:2}
	Equation~\eqref{1} has a unique solution iff $d=1$ and $m_1=1$.
In particular, the homogeneous equation
\begin{equation}\label{homeq}
BX_{n+1}=AX_n,\quad n\in\mathbb{Z}_+
\end{equation}
has only trivial solution iff $d=1$ and $m_1=1$.
\end{corollary}

\begin{corollary}\label{solv:cor:1}
	Equation~\eqref{1} has a unique solution iff $A$ is invertible and $B$ is nilpotent. Moreover, this solution has the form \eqref{gen2}.
	\end{corollary}

\begin{proof}
    According to Corollary~\ref{solv:cor:2}, Equation~\eqref{1} has a unique solution if and only if $d=1$ and $m_1=1$.

    Hence, it suffices to prove that the conditions $B$ is nilpotent and $A$ is invertible are collectively equivalent to the conditions $d=1$ and $m_1=1$.

     At first, prove the sufficiency of the mentioned statement: let $B$ be nilpotent, and $A$ be invertible. Let us proof that $d=1$, $m_1=1$. By the statement~3 of Lemma~\ref{lemels}, the nilpotency of $B$ implies $m_1=1$. If $A$ is invertible, then ${\rm gcd}(a, m)=1$, and hence $d=1$.

    Now prove the inverse statement. Let $d=1$ and $m_1=1$. By the statement~3 of Lemma~\ref{lemels}, the condition $m_1=1$ yields $B$ is nilpotent. By Lemma~\ref{lemels}, if $d=1$, then $A_2$ is an invertible element of $\mathbb{Z}_{m_2}$. Since $m_1=1$, this implies $A_2=A$ is invertible.
    The representation~\eqref{gen2} for the unique solution of Equation~\eqref{1} follows from Theorem \ref{solvability:equation}.
\end{proof}

\begin{corollary}\label{solv:cor:4}
If the homogeneous equation \eqref{homeq}
has only trivial solution, then for any sequence $\{F_n\}_{n=0}^{\infty}$
Equation~\eqref{1} has a unique solution.
Moreover, the unique solution of Equation \eqref{1} has the form \eqref{gen2}.		
\end{corollary}
\begin{proof}
Let Equation \eqref{homeq} has only trivial solution. Then Corollary~ \ref{solv:cor:2} implies $d=1$, $m_1=1$ and, therefore, for any sequence $\{F_n\}_{n=0}^{\infty}$ of $\mathbb{Z}_m$ Equation \eqref{1} has  a unique solution.
The form \eqref{gen2} for the unique solution of Equation \eqref{1} follows from Corollary \ref{solv:cor:1}.
\end{proof}

\section{Examples}

\begin{example}\label{ex:1}
	Consider the following equation over $\mathbb{Z}_6$:
	\begin{equation}\label{ex:1:eq}
		[3]_6X_{n+1}=[2]_6X_n+F_n,\quad n\in\mathbb{Z}_+.
	\end{equation}
	
	Let $F_n=[f_n]_6$. Determine the values: $A=[2]_6$, $B=[3]_6$, $m=6$. Let $b=3$, $a=2$, hence $d=1$. On evidence, $m_1=2$ and $m_2=3$. Also determine: $A_2=[2]_3$, $B_2=[3]_3$, ${\rm ind}(B_2)=1$. Let $Y_0=[y_0]_6$. By the statement~1 of Theorem~\ref{solvability:ip}, the initial problem~\eqref{ex:1:eq}, \eqref{ic} has a solution if and only if
	\begin{equation*}
		[y_0]_3 = - \sum_{s=0}^{{\rm ind}(B_2)-1}A_2^{-s-1}B_2^sF_{2,s}=
		-[2]_3^{-1}[f_0]_3=[f_0]_3,
	\end{equation*}
	i.~e.
	\begin{equation}\label{ex:1:cond:has:sol}
		[y_0]_3=[f_0]_3.
	\end{equation}
	
	Further assume that the solution of the initial problem~\eqref{ex:1:eq}, \eqref{ic} exists, i.~e. the equality~\eqref{ex:1:cond:has:sol} is fulfilled. This solution is unique.
	
	The representation of this solution may be found by the formula~\eqref{solvability:ip:d:eq1:solexpl}.
	Evaluate: $A_1=[2]_2$, $B_1=[3]_2$, $E_1=[3]_2^{-1}=[1]_2$, $E_2=[2]_3^{-1}=[2]_3$ (see also the formulas~\eqref{psi:Ei:Def}). Choose $e_1=1$, $e_2=2$. According to the formula~\eqref{psi:def}, the isomorphism $\psi\colon \mathbb{Z}_2\oplus\mathbb{Z}_3\rightarrow \mathbb{Z}_6$ is defined as follows:
	\begin{equation}\label{ex:1:psi:def}
		\psi(T_1, T_2)=[3t_1+4t_2]_6, \quad \forall T_1=[t_1]_2,\ \forall T_2=[t_2]_3.
	\end{equation}

	Evaluate
\begin{equation}\label{ex:1:psi:arg1}	
		B_1^{-n}A_1^nX_{1,0} + \sum_{s=0}^{n-1}A_1^sB_1^{-s-1}F_{1,n-s-1}=
		[3]_2^{-n}[2]_2^nX_{1,0} + \sum_{s=0}^{n-1}[2]_2^s[3]_2^{-s-1}F_{1,n-s-1}=
		F_{1,n-1}, \quad n\in\mathbb{N},
	\end{equation}
	\begin{equation}\label{ex:1:psi:arg2}
		- \sum_{s=0}^{{\rm ind}(B_2)-1}A_2^{-s-1}B_2^sF_{2,n+s}=
		- [2]_3^{-1}F_{2,n}=F_{2,n}, \quad n\in\mathbb{Z}_+.
	\end{equation}
	
	Substituting~\eqref{ex:1:psi:def}, \eqref{ex:1:psi:arg1} and \eqref{ex:1:psi:arg2} into \eqref{solvability:ip:d:eq1:solexpl}, we obtain the following form for the unique solution of the initial problem~\eqref{ex:1:eq}, \eqref{ic} :
	\begin{equation}
		X_0=Y_0, \quad X_n=\psi(
			[f_{n-1}]_2,
			[f_n]_3		) =
		[3f_{n-1}+4f_n]_6=3F_{n-1}+4F_n, \quad n\in\mathbb{N}.
	\end{equation}
	
	By Theorem~\ref{solvability:equation},  Equation~\eqref{ex:1:eq}  has $m_1=2$ solutions, and the general solution of this equation has the form:
	$$
		X_0 = \psi\left([\beta]_2, [f_0]_3\right)=[3\beta+4f_0]_6,
	$$
	$$
		X_n =\psi\left(
			[f_{n-1}]_2,
			[f_n]_3		\right)=
		[3f_{n-1}+4f_{n}]_6=3F_{n-1}+4F_n, \quad n\in\mathbb{N},
	$$
where 	$\beta $ may be equal to 0 or 1.
\end{example}

\begin{example}\label{ex:2}
	Consider the following equation over $\mathbb{Z}_9$:
	\begin{equation}\label{ex:2:eq}
		[3]_9X_{n+1}=[2]_9X_n+F_n,\quad n\in\mathbb{Z}_+.
	\end{equation}
	
	Let $Y_0=[y_0]_9, F_n=[f_n]_9$. Determine the values: $A=[2]_9$, $B=[3]_9$, $m=9$. Let $b=3$, $a=2$, hence $d=1$. On evidence, $m_1=1$ and $m_2=9$. Here $B$ is nilpotent and $A$ is invertible elements of $\mathbb{Z}_9$. Evaluate: ${\rm ind}(B)=2$.  By Corollary \ref{solv:cor:1}, Equation~\eqref{ex:2} has a unique solution. This solution has the form
\begin{equation}\label{ex:2:formula}
		X_n=-\sum_{s=0}^{{\rm ind}(B)-1}[2]_9^{-s-1}[3]_9^sF_{n+s}=
				-[5]_9F_{n}
		-[25]_9[3]_9F_{n+1}
		=4F_{n}+6F_{n+1}
		, \quad n\in\mathbb{Z}_+.
	\end{equation}
	
	The initial problem~\eqref{ex:2:eq}, \eqref{ic} has a solution if and only if $Y_0=4F_0+6F_{1}$. This solution is unique and has the form~\eqref{ex:2:formula}.
\end{example}

\begin{example}\label{ex:3}
	Consider the following equation over $\mathbb{Z}_{12}$:
	\begin{equation}\label{ex:3:eq}
		[6]_{12}X_{n+1}=[2]_{12}X_n+F_n,\quad n\in\mathbb{Z}_+.
	\end{equation}
	
	Let $F_n=[f_n]_{12}$.
	Determine the values: $A=[2]_{12}$, $B=[6]_{12}$, $m=12$, $a=2$, $b=6$. That implies that $d=2$.	If $f_n$ is odd for some $n\in\mathbb{Z}_+$, then by the statement 3 of Theorem \ref{solvability:equation} Equation~\eqref{ex:3:eq} has no solutions.
	
	Further let $f_n$ be even for all $n\in\mathbb{Z}_+$.
	
	Determine: $m'=6$, $m'_1=2$, $m'_2=3$, $B'=[3]_6$, $A'=[1]_6$, $F'_n=\left[\frac{f_n}{2}\right]_6$. Let $Y_0=[y_0]_{12}$. Also, $Y'_0=[y_0]_6$, $B'_1=[3]_2$, $A'_1=[1]_2$, $B'_2=[3]_{3}$, $A'_2=[2]_3$, $F'_{1, n}=\left[\frac{f_n}{2}\right]_2$, $F'_{2,n}=\left[\frac{f_n}{2}\right]_3$. Here ${\rm ind}(B'_2)=1$.
	
	By the statement~2 of Theorem~\ref{solvability:ip}, the initial problem~\eqref{ex:3:eq}, \eqref{ic} has a solution if and only if
	$$
		[y_0]_3=- \sum_{s=0}^{{\rm ind}(B'_2)-1}(A'_2)^{-s-1}(B'_2)^sF'_{2, s}=-F'_{2,0}=2\left[\frac{f_0}{2}\right]_3=[f_0]_3,
	$$
	i.~e.
	\begin{equation}\label{ex:3:y:cond}
		[y_0]_3=[f_0]_3.
	\end{equation}
	
	The corresponding equation~\eqref{1:dash} over $\mathbb{Z}_6$ has the form
	\begin{equation}\label{ex:4:eq:dash}
		[3]_6X'_{n+1}=X'_n+F'_n, \quad n\in\mathbb{Z}_+.
	\end{equation}
	
	Further let~\eqref{ex:3:y:cond} be fulfilled.
	
    By the statement~1 of Theorem~\ref{solvability:ip}, the initial problem~\eqref{ex:4:eq:dash}, \eqref{ic:dash} has a unique solution, which may be obtained by the formula~\eqref{solvability:ip:d:eq1:solexpl}.

	Evaluate:
	
	\begin{multline}\label{ex:3:psi:arg1}
		\left(B'_1\right)^{-n}\left(A'_1\right)^nY'_{1,0} + \sum_{s=0}^{n-1}\left(A'_1\right)^s\left(B'_1\right)^{-s-1}F'_{1,n-s-1}=\\=\left([3]_2\right)^{-n}Y'_{1,0} + \sum_{s=0}^{n-1}\left([3]_2\right)^{-s-1}F'_{1,n-s-1}=Y'_{1,0} + \sum_{s=0}^{n-1}F'_{1,n-s-1},
	\end{multline}
	
	\begin{equation}\label{ex:3:psi:arg2}
		- \sum_{s=0}^{{\rm ind}\left(B'_2\right)-1}\left(A'_2\right)^{-s-1}\left(B'_2\right)^sF'_{2,n+s}=
		- \left[\frac{f_n}{2}\right]_3.
	\end{equation}
	As in Example \ref{ex:1}, the isomorphism $\psi\colon\mathbb{Z}_2\oplus\mathbb{Z}_3\rightarrow\mathbb{Z}_6$ is defined by the formula \eqref{ex:1:psi:def}.    Substituting~\eqref{ex:1:psi:def}, \eqref{ex:3:psi:arg1} and \eqref{ex:3:psi:arg2} into ~\eqref{solvability:ip:d:eq1:solexpl}, we obtain the unique solution of the initial problem problem~\eqref{ex:4:eq:dash}, \eqref{ic:dash}:
	
	\begin{multline}\label{X'}
		X'_0=Y'_0, \quad X'_n=
		\psi\left(
		Y'_{1,0} + \sum_{s=0}^{n-1}F'_{1,n-s-1},
		-\left[\frac{f_n}{2}\right]_2
		\right)=\\=
		\left[
			3\left(y_0 + \sum_{s=0}^{n-1}\frac{f_s}{2}\right) + 4\left(- \frac{f_n}{2}\right)
		\right]_6=
					\left[3y_0 + 3\sum_{s=0}^{n-1}\frac{f_s}{2} + 4 f_n\right]_6
		,
		\quad
		n\in\mathbb{N}.
	\end{multline}
	
Hence, if $f_n$ ($n\in\mathbb{Z}_+$) is even and~\eqref{ex:3:y:cond} is fulfilled, then by the statement~2 of Theorem~\ref{solvability:ip} the initial problem~\eqref{ex:3:eq}, \eqref{ic} has infinitely many solutions. Moreover, the general solution of this initial problem has the following form (see formulas~\eqref{th1:x'n:to:xn} and \eqref{X'}).
	
    \begin{equation}\label{ff}
    	X_0=Y_0,\quad
        X_n=\left[
        	3y_0 + 3\sum_{s=0}^{n-1}\frac{f_s}{2} + 4 f_n+6\alpha_n
        \right]_{12}, \quad n\in\mathbb{N},
    \end{equation}
    where $\{\alpha_n\}_{n=0}^{\infty}$ is an arbitrary sequence of the elements $0$ and $1$.

By the statement~2 of Theorem~\ref{solvability:equation}, Equation~\eqref{ex:3:eq} has infinitely many solutions. Moreover, the general solution of Equation~\eqref{ex:3:eq} has the form (see  formulas~\eqref{lem:xnlocaldeflem} and~\eqref{ff}):
	\begin{equation*}
		X_0=\left[3\beta+4f_0+6\alpha_n\right]_{12},\quad
		X_n=\left[
			3\beta + 3\sum_{s=0}^{n-1}\frac{f_s}{2} + 4 f_n+6\alpha_n
		\right]_{12}
		, \quad n\in\mathbb{N},
	\end{equation*}
	where $\beta$ and $\alpha_n$ ($n\in\mathbb{Z}_+$) are arbitrary elements of $\{0,1\}$.
\end{example}

\begin{example}\label{ex:4}
	Consider the following equation over $\mathbb{Z}_{12}$:
	\begin{equation}\label{ex:4:eq}
[9]_{12}X_{n+1}=[6]_{12}X_n+F_n,\quad n\in\mathbb{Z}_+.
	\end{equation}

Let $Y_0=[y_0]_{12}$, $F_n=[f_n]_{12}$.	
	Determine the values: $A=[6]_{12}$, $B=[9]_{12}$, $m=12$. Let $a=6$, $b=9$. This implies that $d=3$. By the statement~2 of Theorem~\ref{solvability:ip},	if $3\nmid f_n$ for some $n\in\mathbb{Z}_+$, then by the statement 3 of Theorem \ref{solvability:equation} Equation~\eqref{ex:4:eq} has no solutions.

Further let $3\mid f_n$ for all $n\in\mathbb{Z}_+$. Determine $B'=[3]_4$, $A'=[2]_4$, $m'=4$, $m'_1=4$, $m'_2=1$.  The corresponding equation~\eqref{1:dash} over $\mathbb{Z}_4$ has the form:
    \begin{equation}\label{ex:4:eq:prime}
           [3]_4X'_{n+1}=[2]_4X'_n+F'_n,
    		\quad n\in\mathbb{Z}_+,
    \end{equation}
	where $F'_n=\left[\frac{f_n}{3}\right]_4$.
	
	By the first statement of Theorem~\ref{solvability:ip}, for any $Y'_0\in\mathbb{Z}_4$ the initial problem~\eqref{ex:4:eq:prime}, \eqref{ic:dash} has a unique solution and this solution is defined by the formula:
    \begin{multline}
    	X'_0=Y'_0,\quad
    	X'_n=
    	\left(B'\right)^{-n}\left(A'\right)^nY'_0 + \sum_{s=0}^{n-1}\left(A'\right)^s\left(B'\right)^{-s-1}F'_{n-s-1}
    	=\\=
    	[3]_4^{-n}[2]_4^nY'_0 + \sum_{s=0}^{n-1}[2]_4^s[3]_4^{-s-1}F'_{n-s-1}
    	=[2]_4^{n}Y'_0 + \sum_{s=0}^{n-1}[2]_4^s[3]_4^{-s-1}F'_{n-s-1},
    	\quad n\in\mathbb{N}.
    \end{multline}

	More precise,
	\begin{equation*}
		X'_0=Y'_0,\quad X'_1=2Y'_0+3F'_0,\quad X'_n=3F'_{n-1}+2F'_{n-2},\quad n=2,3,\ldots.
	\end{equation*}

	If $3\mid f_n,\ n\in\mathbb{Z}_+$, then by the second statement of Theorem~\ref{solvability:ip}, for any $Y_0\in\mathbb{Z}_{12}$ the initial problem~\eqref{ex:4:eq}, \eqref{ic} has infinitely many solutions. Moreover, the general solution of this initial problem has the following form (see formula~\eqref{th1:x'n:to:xn}):
	$$
		X_0=Y_0,\quad X_n=[x'_n+4\alpha_n]_{12},\quad n\in\mathbb{N},
	$$
	i.~e.
\begin{equation}\label{genex4}					
			X_0=Y_0,\quad X_1=[2y_0+f_0+4\alpha_1]_{12},\quad
X_n=\left[f_{n-1}+2\frac{f_{n-2}}{3}+4\alpha_n\right]_{12},\quad n=2,3,\ldots.
		\end{equation}
	Here $\{\alpha_n\}_{n=0}^{\infty}$ is an arbitrary sequence of $\{0,1,2\}$.

If $3\mid f_n,\ n\in\mathbb{Z}_+$, then by the statement~2 of Theorem~\ref{solvability:equation},  Equation~\eqref{ex:4:eq} has infinitely many solutions. Moreover, the general solution of this equation has the following form (see the formula~\eqref{lem:xnlocaldeflem}).
	\begin{equation}\label{x_nex4}
		X_n=[x'_n+4\alpha_n]_{12}, \quad n\in\mathbb{Z}_+,
	\end{equation}
	where
  $\{\alpha_n\}_{n=0}^{\infty}$ is an arbitrary sequence of $\{0,1,2\}$ and the sequence
$X'_n=[x'_n]_4\ (n\in\mathbb{Z}_+)$ is the general solution of Equation \eqref{ex:4:eq:prime} which is defined as follows:
$$
	  X_0'=[x_0']_4,\quad
		X'_1=2X'_0+3F'_0,\quad X'_n=3F'_{n-1}+2F'_{n-2},\quad n=2,3,\ldots.
	$$
  Now the general solution~\eqref{x_nex4} of Equation~\eqref{ex:4:eq} can be written in a more convenient form, which is  similar to \eqref{genex4}:
		$$X_0=[x_0]_{12},\quad X_1=[2x_0+f_0+4\alpha_1]_{12},\quad  X_n=\left[f_{n-1}+2\frac{f_{n-2}}3+4\alpha_n\right]_{12},\quad n=2,3,\ldots,$$
		where $x_0$ is an arbitrary integer and $\{\alpha_n\}_{n=0}^{\infty}$ is an arbitrary sequence of  $\{0,1,2\}$.

\end{example}

\end{document}